\documentclass[sn-mathphys]{sn-jnl}

\usepackage{amsmath,amsfonts,amssymb}
\usepackage{caption}
\usepackage{subcaption}

\jyear{2021}%

\theoremstyle{thmstyleone}%
\newtheorem{theorem}{Theorem}
\newtheorem{proposition}[theorem]{Proposition}%
\newtheorem{lemma}[theorem]{Lemma}
\newtheorem{open}[theorem]{Open question}

\theoremstyle{thmstyletwo}%
\newtheorem{example}{Example}%

\theoremstyle{thmstylethree}%

\newcommand{\bs}[1]{\boldsymbol{#1}}
\newcommand{\ones}{\mathbf{1}}
\newcommand{\tran}{\mathsf{T}} 	
\newcommand{\diag}{\mathrm{diag}}
\newcommand{\J}{{\mathcal J}}
\newcommand{\ee}{\bs{e}}
\newcommand{\vv}{\boldsymbol{v}}
\newcommand{\uu}{\boldsymbol{u}}
\newcommand{\ww}{\boldsymbol{\vec{w}}}

\begin{document}
\title[Feasibility of sparse ecosystems]{Feasibility of sparse large Lotka-Volterra ecosystems}

\author*[1]{\fnm{Imane} \sur{Akjouj}}\email{imane.akjouj@univ-lille.fr}

\author*[2]{\fnm{Jamal} \sur{Najim}}\email{najim@univ-mlv.fr}


\affil[1]{\orgdiv{Laboratoire Paul Painlev\'e, UMR 8524}, \orgname{Universit\'e de Lille}, \orgaddress{\street{Avenue Paul Langevin, Cit\'e Scientifique}, \city{Villeneuve d'Ascq Cedex}, \postcode{59655}, \country{France}}}

\affil[2]{\orgdiv{Laboratoire d'Informatique Gaspard Monge, UMR 8049}, \orgname{CNRS \& Universit\'e Gustave Eiffel}, \orgaddress{\street{5, Boulevard Descartes, Champs-sur-Marne}, \city{Marne-la-Vall\'ee Cedex 2}, \postcode{77454} \country{France}}}


\abstract{Consider a large ecosystem (foodweb) with $n$ species, where the abundances follow a Lotka-Volterra system of coupled differential equations. We assume that each species interacts with $d=d_n$ other species and that their interaction coefficients are independent random variables. 
	
	This parameter $d$ reflects the connectance of the foodweb and the sparsity of its interactions especially if $d$ is much smaller that $n$.
	
	We address the question of feasibility of the foodweb, that is the existence of an equilibrium solution of the Lotka-Volterra system with no vanishing species. We establish that for a given range of $d$, namely $d\propto n$ or $d\ge \log(n)$ with an extra condition on the sparsity structure, there exists an explicit threshold depending on $n$ and $d$ and reflecting the strength of the interactions, which guarantees the existence of a positive equilibrium as the number of species $n$ gets large.   
	
	From a mathematical point of view, the study of feasibility is equivalent to the existence of a positive solution $\bs{x}_n$ (component-wise) to the equilibrium linear equation: 
	\begin{equation*}
	\bs{x}_{n} = \ones_n + M_n \bs{x}_n \,, 
	\end{equation*}
	
	where $\ones_n$ is the $n\times 1$ vector with components 1 and  $M_n$ is a large sparse random matrix, accounting for the interactions between species.
	The analysis of such positive solutions essentially relies on large random matrix theory for sparse matrices and Gaussian concentration of measure. The stability of the equilibrium is established.
	
	The results in this article extend to a sparse setting the results obtained by Bizeul and Najim in \cite{bib4}.
	}




\maketitle

\section{Introduction}\label{sec:intro} 

\subsection*{Lotka-Volterra system of coupled differential equations.} Large Lotka-Volterra (LV) systems are widely used in mathematical biology and ecology to model pobulations with interactions \cite{bib26,bib25,bib27}. 

For a given foodweb, denote by $ \bs{x}_n=(x_{k}(t))_{1\leq k\leq n} $ the vector of abundances of the various species at time $t\geq 0$. 
In a LV system, the abundances are connected via the following coupled equations:
$$
\frac{dx_k(t)}{dt} = x_k(t)\, \left( r_k - x_k(t) +  \sum_{\ell=1}^n M_{k\ell} x_{\ell}(t)\right)\qquad \textrm{for}\quad k\in [n]:=\{1,\cdots, n\}\, ,
$$
where $M_n=(M_{k\ell})$ stands for the interaction matrix, and $r_k$ for the intrinsic growth of species $k$.
At the equilibrium $\frac{d\bs{x}_n}{dt}=0$, the abundance vector $\bs{x}_n=(x_k)_{k\in [n]}$ is solution of the system:
\begin{equation}\label{eq:equilibrium}
x_k\, \left( r_k - x_k +  \sum_{\ell\in[n]} M_{k\ell} x_{\ell}\right)=0\qquad \textrm{for}\quad x_k\ge 0\quad \text{and}\quad k\in [n]\ .
\end{equation}

An important question, which motivated recent developments \cite{bib7,bib4}, is the existence of a {\it feasible} solution $\bs{x}_n$ to \eqref{eq:equilibrium}, that is a solution where all the $x_k$'s are positive, corresponding to a scenario where no species disappears. Notice that in this latter case, the system \eqref{eq:equilibrium} takes the much simpler form:
$$
\bs{x}_n = \boldsymbol{r}_n + M_n \bs{x}_n\, ,
$$
where $\boldsymbol{r}_n=(r_k)$.

Aside from the question of feasibility arises the question of \textit{stability} : for a complex system, how likely a perturbation of the solution $\bs{x}_n$ at equilibrium  will return to the equilibrium? Gardner and Ashby \cite{bib8} considered stability issues of complex systems connected at random. Based on the circular law for large random matrices with i.i.d. entries, May \cite{bib15} provided a complexity/stability criterion and motivated the systematic use of large random matrix theory in the study of foodwebs, see for instance  Allesina et al. \cite{bib2}. Recently, Stone \cite{bib16} and Gibbs et al. \cite{bib11} revisited the relation between feasibility and stability. 

In the spirit of May\footnote{Beware that May did not consider LV systems but rather used a random matrix model for the Jacobian at equilibrium of a generic system of coupled differential equations.} and in the absence of any prior information, we shall model the interactions of matrix $M_n$ as random and in order to simplify the analysis, we will consider intrinsic growths $(r_i)_{i \in [n}$ equal to 1, and the equations under study will take the following form in the sequel:
\begin{equation}\label{eq:LV}
\frac{dx_k(t)}{dt} = x_k(t)\, \left( 1 - x_k(t) +  \sum_{\ell\in[n]} M_{k\ell} x_{\ell}(t)\right)\qquad \textrm{for}\quad k\in [n]\, .
\end{equation}

\subsection*{Sparse foodwebs}  One of the most important parameters of the complexity of an ecosystem is its connectance, which is the proportion of interactions between species (see for instance \cite{bib20}). This corresponds to the proportion of non-zero entries in the interaction matrix $M_n$. May's complexity/stability criterion asserts that the instability of an ecosystem increases with the connectance (i.e. the less sparse $M_n$ is, the more unstable is the ecosystem equilibrium). More specifically, \cite{bib21} specifies that the effect of the sparsity depends on the nature of the interactions (random, predator-prey, mutualistic or competitive). In the case of random interactions, \cite{bib23} supports the idea that sparse ecosystems lead to a stable equilibrium. Based on ecological and biological data (see for instance \cite{bib22}), recent studies \cite{bib6} suggest that foodwebs can actually be very sparse. In a recent theoretical study, \cite{bib31} study the properties of sparse ecological communities in relation with the strength of interactions.

To encode this sparsity in a simple parametric way, we first consider a directed $d_n$-regular graph with $n$ vertices and its associated 
$n\times n$ adjacency matrix $\Delta_n=(\Delta_{ij})$:  
$$
\Delta_{ij} = \begin{cases}
1 & \text{if there is an edge pointing from}\ i\ \text{to}\ j\,,\\
0 & \text{otherwise.}
\end{cases}
$$
In the considered graph, each vertex $i$ has $d_n$ edges pointing from a vertex $k\in [n]$ to $i$, and has $d_n$ other edges pointing from $i$ to a vertex $\ell\in [n]$. An edge pointing from $i$ to $i$ is called a loop. 
In particular, matrix $\Delta_n$ is deterministic, has exactly $d_n$ non-null entries per row and per column, and $n\times d_n$ non-null entries overall. 

Denote by $A_n$ a $n\times n$ matrix with independent Gaussian ${\mathcal N}(0,1)$ entries and consider the Hadamard product matrix
$
\Delta_n\circ A_n =(\Delta_{ij} A_{ij})
$. Let $(\alpha_n)_{n\ge 1}$ be a positive sequence. We assume that matrix $M_n$ has the following form
\begin{equation}\label{eq:model-M}
M_n = \frac{\Delta_n \circ A_n}{\alpha_n \sqrt{d_n}}\, .
\end{equation}
Let us comment on the normalizing factor $1/(\alpha_n \sqrt{d_n})$. Theoretical results on sparse large random matrices \cite{bib1} assert that asymptotically
$$
\left\| \frac{\Delta_n\circ A_n}{\sqrt{d_n}}\right\| ={\mathcal O}(1)\,,\quad (n\to\infty) 
$$
where $\|\cdot\|$ stands for the spectral norm, if the degree $d_n$ of the graph satisfies $d_n\ge \log(n)$, a condition that we will assume in the remaining of the article. In particular, normalization $1/\sqrt{d_n}$ guarantees that matrix $\Delta_n\circ A_n/\sqrt{d_n}$ has a macroscopic effect in the LV system, even for large foodwebs (large $n$).

The extra normalization $1/\alpha_n$ is to be tuned to get a feasible solution. 

Denote by $\boldsymbol{1}_n$ the $n\times 1$ vector of ones and by $A^{\tran}$ the transpose of matrix $A$. In the full matrix case $\Delta_n =\ones_n \ones_n^{\tran}$, \cite{bib7}, based on  \cite{bib10}, proved that a feasible solution is very unlikely to exist if $\alpha_n\equiv\alpha$ is a constant. We thus consider the regime where $\alpha_n\to\infty$ and will prove that there is a sharp threshold $\alpha_n\sim \sqrt{2\log(n)}$ above which a feasible solution exists (with high probability) and below which does not. This phase transition has already been established in \cite{bib4} for the full matrix case.

One can notice that, in sparse foodwebs ($d_n<n$), the interaction coefficients can be stronger than when the interaction matrix is full (i.e. when $d_n=n$) in the sense that $\frac{1}{\sqrt{d_n}}>\frac 1{\sqrt{n}}$. 

\subsection*{Models and feasibility results} The sparse random matrix model under investigation is given in \eqref{eq:model-M}.
Specifying the range of $d_n$ and the structure of $\Delta_n$, we introduce hereafter two models amenable to analysis.

\subsubsection*{\textbf{Model (A)}: Block permutation matrix.} 
Let $n=d\times m$. Denote by ${\mathcal S}_m$ the group of permutations of $[m]=\{1,\dots, m\}$. Given $\sigma\in \mathcal{S}_m$, consider the associated permutation matrix 
$$
P_{\sigma}=(P_{ij})_{i,j\in [m]} \quad \text{where}\quad P_{ij}=\begin{cases}
1&\textrm{if}\ j=\sigma(i),\\
0&\textrm{else.}
\end{cases}
$$
Denote by $J_d={\bf 1}_d {\bf 1}_d^{\tran}$ the $d\times d$ matrix of ones.
Assume that 
\begin{itemize}[label=-]
	\item matrix $M_n$ is given by \eqref{eq:model-M},
	\item $d=d_n\ge \log(n)$,
	\item matrix $\Delta_n$ introduced in \eqref{eq:model-M} is a {\it block-permutation} adjacency matrix given by 
	\begin{equation}\label{eq:block-permutation}
	\Delta_n = P_{\sigma}\otimes J_d = \left( P_{ij} J_d\right)_{i,j\in [m]}\, ,
	\end{equation}
	where $\otimes$ is the Kronecker matrix product.
\end{itemize}


Notice that $\Delta_n$ still corresponds to the adjacency matrix of a $d$-regular graph.

\begin{example}\label{example} To illustrate these definitions, we provide an example. Let $n=m\times d$ with $m=4$ and $\sigma\in {\mathcal S}_4$ defined by
	$$
	\sigma=\begin{pmatrix}
	1&2&3&4\\
	1&4&2&3
	\end{pmatrix}\ .
	$$
	Matrices $P_\sigma$, $\Delta$ and $\Delta\circ A$ are respectively given by:
	$$
	P_\sigma =  \begin{pmatrix} 1 & 0 & 0 & 0 \\ 0 & 0 & 0 & 1 \\ 0 & 1 & 0 & 0 \\ 0 & 0 & 1 & 0  \end{pmatrix}\,,\quad 
	\Delta =\begin{pmatrix}
	J_d & 0 & 0 & 0 \\ 0 & 0 & 0 & J_d \\ 0 & J_d & 0 & 0 \\ 0 & 0 & J_d & 0 
	\end{pmatrix}\,,\quad
	\Delta \circ A =  \begin{pmatrix} A^{(1)} & 0 & 0 & 0 \\ 0 & 0 & 0 & A^{(2)} \\ 0 & A^{(3)} & 0 & 0 \\ 0 & 0 & A^{(4)} & 0  \end{pmatrix}\, ,
	$$
	where $A^{(\mu)}$ ($\mu\in[4]$) is a $d\times d$ matrix with i.i.d. ${\mathcal N}(0,1)$ entries.
\end{example}

\subsubsection*{\textbf{Model (B)}: $d$ is proportional to $n$.}
Assume that $M_n$ is given by \eqref{eq:model-M} and that $d=d_n$ satisfies
\begin{equation}\label{eq:dpropn}
\lim_{n\to\infty} \frac{d_n}{n} = \beta>0\, .
\end{equation}

We can now state the main result of the article:
\begin{theorem}\label{th:sparse-matrix}
	Let $A_n$ be a $n\times n$ matrix with i.i.d. ${\mathcal N}(0,1)$ entries and $\Delta_n$ given by Model (A) or (B).
	Assume that $\alpha_n \xrightarrow[n\to\infty]{} \infty$  and denote by $$
	\alpha_n^*=\sqrt{2\log n}\ .
	$$ 
	Let $\bs{x}_{n}=(x_k)_{ k\in [n] }$ be the solution of 
	\begin{equation}\label{eq:equilibrium-A}
	\bs{x}_{n} = \ones_n + \frac{1}{\alpha_n\sqrt{d_n}} \left( \Delta_n \circ A_n\right) \bs{x}_{n}\ .
	\end{equation}
	Then
	\begin{enumerate}
		\item If $\exists\, \varepsilon>0$ such that eventually $\alpha_n\le (1-\varepsilon) \alpha_n^*$ then 
		$$
		\mathbb{P}\left\{ \min_{k\in [n]} x_k>0\right\} \xrightarrow[n\to\infty]{} 0 \, ,
		$$		
		\item If $\exists\, \varepsilon>0$ such that eventually $\alpha_n\ge (1+\varepsilon)\alpha_n^*$ then
		$$
		\mathbb{P}\left\{ \min_{k\in [n]} x_k>0\right\} \xrightarrow[n\to\infty]{} 1\, .
		$$
	\end{enumerate}
\end{theorem}

The results of Theorem \ref{th:sparse-matrix} are illustrated in Fig. \ref{fig:simu-model-A}.

\subsubsection*{Remarks} 
\begin{enumerate}
	\item  By taking $d_n\ge \log(n)$, we guarantee that the spectral norm of matrix $\frac{\Delta_n \circ A_n}{\sqrt{d_n}}$ is of order ${\mathcal O}(1)$, see \cite{bib1}. In particular, matrix $\left( I_n - \frac{\Delta_n \circ A_n}{\alpha_n \sqrt{d_n}}\right)$ is invertible and the solution $\bs{x}_n$ can be represented as:
	$$
	\bs{x}_n = \left( I_n - \frac{\Delta_n \circ A_n}{\alpha_n \sqrt{d_n}}\right)^{-1} \ones_n\ .
	$$
	
	\item An informal first-order expansion of the solution immediatly explains this phase transition.
	If we expand the inverse matrix and neglect the remaining terms, we get
	$$
	\bs{x}_n \ \simeq\ \ones_n + \frac{\Delta_n \circ A_n}{\alpha_n \sqrt{d_n}}\ones_n \ =\ \ones +\frac{{\bs z}_n}{\alpha_n}
	$$
	where
	$$
	{\bs z}_n= (z_i)\quad \text{and}\quad z_i =\sum_{j=1}^n \frac{(\Delta_n \circ A_n)_{ij}}{\sqrt{d_n}}\, .
	$$
	Notice that the $z_i$'s remain i.i.d. ${\mathcal N}(0,1)$. Going one step further in the approximation yields $$\min_{i\in [n]} x_i\ \simeq\ 1+\frac{\min_{i\in[n]} z_i}{\alpha_n}\, .$$ 
	By standard extreme value results, we have $\min_{i\in [n]} z_i \sim -\sqrt{2\log(n)}$, hence the phase transition.
	\\
	\item The component-wise positivity of the solution has been studied in the full matrix case, i.e. $\Delta_n =\ones_n \ones_n^{\tran}$ and $d_n =n$, in \cite{bib4} where the same phase transition phenomenon occurs. Proof of Theorem \ref{th:sparse-matrix} can be handled as in \cite{bib4} for Model (B) with non-trivial adaptations that will be specified.
	
	In the case where $d_n \ll n$, a normalization issue occurs. To say it roughly, the Euclidian norm of vector $\ones_n/\sqrt{d_n}$ is no longer of order ${\mathcal O}(1)$ but of order $\sqrt{n/d_n}$ and one needs to handle more carefully the sparsity of matrix $\Delta_n$.  
	
	In this regard, the block-permutation structure of Model (A) is a technical and simplifying assumption. The problem of the component-wise positivity of $\bs{x}_n$ for a general adjacency matrix $\Delta_n$ of a $d$-regular graph with $d\ge \log(n)$ remains open.
\end{enumerate}

\begin{figure}[h]
	\centering
	\includegraphics[scale=0.7]{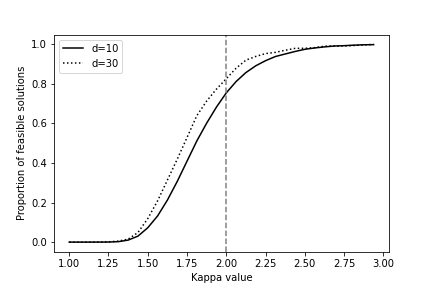}
	\caption{Let $n=15000$ with $\log(n)\simeq 9.61$. For $d=10$ and $m=1500$, we first draw at random a permutation $\sigma\in {\mathcal S}_m$ and fix $\Delta_n=P_{\sigma} \otimes \bs{1}_d\bs{1}_d^\tran$ once for all. Each point of the solid line represents the proportion of feasible solutions $\bs{x}_n$ of \eqref{eq:equilibrium-A} over $2000$ realizations of random matrices $A_n$ for different values of $\kappa$, with $\alpha_n=\sqrt{\kappa \log(n)}$. The same simulation is realized with $d=30$ over 500 realizations of $A_n$ (dotted line).} \label{fig:simu-model-A}
\end{figure}

\subsection*{Stability results} 
A classical property of \eqref{eq:LV} is the positivity of the orbits\footnote{Beware that this property does not prevent some components $x_i(t)$ to converge to zero, hence does not enforce a feasible equilibrium.}: if $\bs{x}_n^0\in (\mathbb{R}^{*+})^n$, then $\bs{x}_n^t\in (\mathbb{R}^{*+})^n$ as well $(t>0)$.

We first recall definitions related to stability from \cite[Chapter 3]{bib19}. An equilibrium $\bs{x}_n$ is \textbf{\textit{stable}} if for any given neighborhood $W$ of $\bs{x}_n$, there exists a neighborhood $V$ such that for any initial point $\bs{x}_n^0\in V$, the orbit $\{\bs{x}_n^t;\ t\ge 0;\ \bs{x}_n^0\in V\}$ stays in $W$. In addition, if the equilibrium is stable and the orbit converges to $\bs{x}_n$, the equilibrium is said \textbf{\textit{asymptotically stable}}.

In the full matrix case ($\Delta_n =\ones_n \ones_n^{\tran},\ d_n =n$), it has been proved in \cite{bib4} that in the regime where feasibility occurs, the system is asymptotically stable in the sense that the Jacobian matrix ${\mathcal J}$ of the LV system \eqref{eq:LV} evaluated at $\bs{x}_n$:
\begin{equation}\label{eq:jacobian}
\J(\bs{x}_n) = \diag(\bs{x}_n)\left(-I_n + M_n\right)
\end{equation}
has all its eigenvalues with negative real part.

Finally, the equilibrium is \textbf{\textit{globally stable}}  when it is asymptotically stable and the neighborhood $V$ can be taken as the whole state place $(\mathbb{R}^{*+})^n$. 

We complement Theorem \ref{th:sparse-matrix} and prove that feasibility and global stability occur simultaneously.

\begin{theorem}[Global stability, Takeuchi and Adachi {\cite[Theorem 3.2.1]{bib19}}]\label{th:global stability} 
	Let $d_n\ge \log(n)$, $\alpha_n \xrightarrow[n\to\infty]{} \infty$, and $\Delta_n$ the adjacency matrix of a $d_n$-regular graph. Then, with probability going to one as $n\to \infty$, Eq. \eqref{eq:equilibrium} admits a unique nonnegative solution $\bs{x}_n$. Moreover, this solution is a globally stable equilibrium. 
\end{theorem}

Beware that in this theorem, the solution, although unique, is no longer (component-wise) positive and may have zero components corresponding to vanishing species. Notice that the assumption over $\Delta_n$ covers Models (A) and (B) but is far less restrictive. We illustrate Theorem \ref{th:global stability} in Fig. \ref{fig:vanish}.

\begin{figure}[h]
	
	\begin{subfigure}{.46\linewidth}\label{subfig:A}
		\includegraphics[scale=0.4]{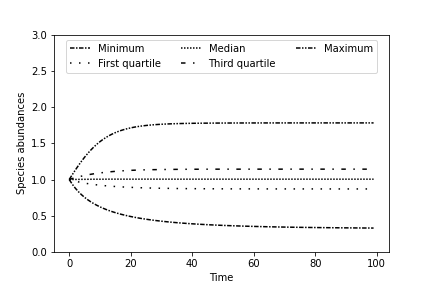}
		\caption{Feasible equilibrium for $\alpha_n=\sqrt{3\log(n)}$. }
	\end{subfigure}%
	\hspace*{\fill}   
	\begin{subfigure}{0.46\textwidth}\label{subfig:B}
		\includegraphics[scale=0.4]{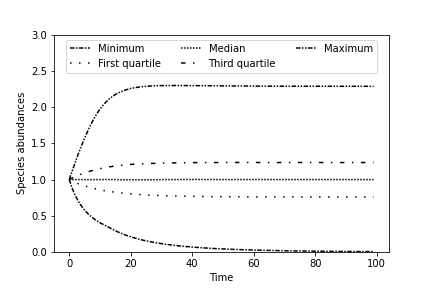}
		\caption{Vanishing species for $\alpha_n=\sqrt{\log(n)}$. }
	\end{subfigure}%

	\caption{LV system with feasible equilibrium (left) and vanishing species (right): minimum, maximum and mean of the population dynamics $(\bs{x}_n^t,\, t>0)$ solution of \eqref{eq:LV} for $n=5000$ ($\log(n)\simeq 8.51$), $d=10$ and $\Delta_n$ follows Model (A). In the first figure, $\alpha_n>\sqrt{2\log(n)}$, the minimum abundance remains positive. In the second one, $\alpha_n<\sqrt{2\log(n)}$, the minimum abundance vanishes and the equilibrium is not feasible.}   \label{fig:vanish}
\end{figure}

We now specify Theorem \ref{th:global stability} in the case of feasibility. 


\begin{proposition}[Stability and convergence rate]\label{prop:stability} 
	Let $d_n\ge \log(n)$, $\alpha_n \xrightarrow[n\to\infty]{} \infty$, and assume that $\Delta_n$ is given by Model (A) or (B). Denote by $\Sigma_n$ the spectrum of the Jacobian matrix $\J(\bs{x}_n)$ given by \eqref{eq:jacobian}.
	
	Assume that there exists $\varepsilon>0$ such that eventually $\alpha_n\ge (1
	+\varepsilon) \alpha_n^*$. Then: 
	\begin{enumerate}
		\item The probability that the equilibrium $\bs{x}_n$ is feasible and globally stable converges to 1,
		\item The spectrum $\Sigma_n$ asymptotically coincides with $-\diag(\bs{x}_n)$ in the sense that:
		$$
		\max_{\lambda\in \Sigma_n}  \min_{k\in [n]} \left\lvert \lambda + x_k\right\rvert \ \xrightarrow[n\to\infty]{\mathcal P}\ 0\, ,
		$$    
		\item Moreover,
		\begin{equation}\label{eq:stability}
		\max_{\lambda\in \Sigma_n}\mathrm{Re}(\lambda) \ \le\  -(1-\ell^+) + o_P(1)\qquad \textrm{where}\qquad 
		\ell^+:=\limsup_{n\to \infty}\frac{\alpha^*_n}{\alpha_n}<1\, .
		\end{equation}
	\end{enumerate}
\end{proposition}
As a consequence of \eqref{eq:stability}, for any $\bs{x}_n^0\in (\mathbb{R}^{+*})^n$, the orbit $\bs{x}_n^t$ converges to the equilibrium $\bs{x}_n$ at an exponential convergence rate, see Fig. \ref{fig:convergence-equilibrium}-(A).

\begin{figure}[h]
	\begin{subfigure}{.46\linewidth}
		\includegraphics[scale=0.4]{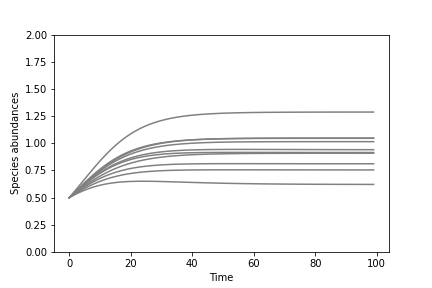}
		\caption{Population dynamics with starting abundances equals to $\frac{1}{2}$. }
	\end{subfigure}%
	\hspace*{\fill}   
	\begin{subfigure}{.46\linewidth}
		\includegraphics[scale=0.4]{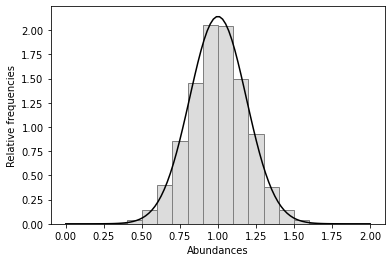}
		\caption{Histogram of the equilibrium abundances.}
	\end{subfigure}%
	\caption{Consider the population dynamics $(\bs{x}_n^t,\, t>0)$ solution of \eqref{eq:LV} where $M$ is given by \eqref{eq:model-M} and $\Delta_n$ follows Model (A) with $n=15000$ species, $m=1500$ blocks, $d=10>\log(n)\simeq 9.62$ and $\alpha_n=\sqrt{3\log(n)}$. On the left, we plot $10$ species randomly chosen over 15000 with starting abundances equals to $\frac{1}{2}$. On the right, the histogram of the abundances is represented, and the normal density with mean $1$ and variance $\frac{1}{\alpha_n^2}$ is fitted. Notice the substantial spread of the abundances despite the high value of $n$.}  
	\label{fig:convergence-equilibrium}
\end{figure}

\subsection*{Notations} If $\boldsymbol{v}$ is a vector then $\|\boldsymbol{v}\|$ stands for its Euclidian norm; if $A$ is a matrix then $\|A\|$ stands for its spectral norm and $\|A\|_F=\sqrt{\sum_{ij} \lvert A_{ij} \rvert ^2}$ for its Frobenius norm. Let $\varphi$ be a function from some space ${\mathcal X}$ (usually $\mathbb{R}$) to $\mathbb{R}$ then $\|\varphi\|_{\infty} =\sup_{x\in {\mathcal X}} \lvert\varphi(x)\rvert$. Convergence in probability is denoted by $\xrightarrow[]{\mathcal P}$.
When no confusion can occur, we shall drop $n$ and simply denote $A, \Delta, \alpha, d,\bs{x}, \text{etc.}$ instead of $A_n, \Delta_n, \alpha_n, d_n,\bs{x}_n, \text{etc}$.

\subsection*{Organization of the paper} 
In Section \ref{sec:spectral-norm}, the spectral norm of a sparse matrix and the general strategy of proof are described.
Proof of Theorem \ref{th:sparse-matrix} is provided in Section \ref{section:proof-main} for Model (A), and in Section \ref{section:proof-main-B} for Model (B). Theorem \ref{th:global stability} is proved in Section \ref{section:proof-stability}. In Section \ref{sec:conclusion}, we conclude and state an open question.

\subsection*{Acknowlegments} 
The authors thank Maxime Clénet, François Massol and Mylène Maïda for fruitful discussions and are grateful to Nick Cook for his insight on the singular values of a sparse random matrix (see Appendix \ref{app:simplicity}).

\section{Spectral norm of the interaction matrix and strategy of proof}\label{sec:spectral-norm}

\subsection{The spectral norm of $\Delta_n\circ A_n/\sqrt{d}$}

In the following proposition which proof is based on \cite{bib1}, we provide an estimate of $\| \Delta\circ A/\sqrt{d}\|$. The fact that $A$'s entries are ${\mathcal N}(0,1)$ and that $d_n\ge \log(n)$ is crucial.

\begin{proposition}\label{prop:spectral-norm} Assume that $A$ is a $n\times n$ matrix with i.i.d. ${\mathcal N}(0,1)$ entries, that $\Delta$ is a $n\times n$ adjacency matrix of a $d$-regular graph, that $d\ge \log(n)$. Then there exists a constant $\kappa>0$ independent from $n$ (one can take for instance $\kappa=22$) such that 
	$$
	\mathbb{P}\left( \left\| \frac{\Delta \circ A}{\sqrt{d}}\right\| \ge \kappa\right) \xrightarrow[n\to \infty]{} 0\, .
	$$
	In particular, let $\delta\in (0,1)$ be fixed and $\alpha=\alpha(n) \xrightarrow[n\to\infty]{}\infty$. Then
	$$
	\mathbb{P}\left( \left\| \frac{\Delta\circ A}{\alpha\sqrt{d}}\right\| \le 1-\delta\right) \xrightarrow[n\to\infty]{}1\, .
	$$
\end{proposition}

\begin{proof}
	Applying \cite[Corollary 3.11]{bib1} to $\frac{\Delta \circ A}{\sqrt{d}}$ with $\epsilon=\frac{1}{2}$, we obtain 
	$$
	\mathbb{P} \left( \left\| \frac{\Delta \circ A}{\sqrt{d}}\right\| \ge 3 + \frac{15}{2 \sqrt{\log \frac{3}{2}}} \times \frac{\sqrt{\log n}}{\sqrt{d}} + \frac{t}{\sqrt{d}}\right) \le e^{-\frac{t^2}{2}} \, .
	$$
	Fix $t=\sqrt{\log n}$, then $e^{-\frac{t^2}{2}} = \frac{1}{\sqrt{n}} \xrightarrow[n\to\infty]{} 0$ and $\frac{t}{\sqrt{d}} \le 1$ by assumption. 	Furthermore, there exists a rank $n_1$ such that for all $n \ge n_1$ :
	$$
	3+\frac{15}{2\sqrt{\log (3/2)}} \times \frac{\sqrt{\log n}}{d} + \frac{t}{\sqrt{d}} \quad \le \quad 4 + \frac{15}{2\sqrt{\log \frac{3}{2}}}\quad  <\quad  \kappa:=22 \, .
	$$
	Thus, 
	$\mathbb{P}\left( \left\| \frac{\Delta \circ A}{\sqrt{d}}\right\| \ge \kappa \right) \to 0$. 
	Since $\alpha\to \infty$, the last part of the proposition immediatly follows.
\end{proof}

\subsection{Strategy of proof} \label{subsection:The resolvent} Based on the previous control of the spectral norm in probability, we reduce the problem of feasibility to the control of the extreme values of high order terms of the resolvent, considered as a Neumann sum, see Lemma \ref{lemma:main}. This preliminary step is similar to \cite[Section 2.1]{bib4}.

Going back to Eq.\eqref{eq:equilibrium-A}, we can write $\left( I - \frac{\Delta\circ A}{\alpha\sqrt{d}}\right)\bs{x}=\ones$. Introducing the resolvent $Q=\left( I - \frac{\Delta\circ A}{\alpha\sqrt{d}}\right)^{-1}$ which by Proposition \ref{prop:spectral-norm} exists with probability tending to one, we obtain the representation
$$
\bs{x}=(x_k)_{k}=\left( I - \frac{\Delta\circ A}{\alpha\sqrt{d}}\right)^{-1}\ones = Q \ones
$$
which holds with growing probability. Denote by $\ee_k$ the $n\times 1$ $k$-th canonical vector, then $x_k=\ee_k^{\tran} \bs{x} = \ee_k^{\tran}Q\ones$. Unfolding the resolvent as a Neumann sum, we obtain
\begin{equation} \label{eq:decomp}
x_k \quad=\quad \ee_k^{\tran} Q\ones \quad =\quad \sum_{\ell =0}^\infty \ee_k^{\tran} \left( \frac{\Delta\circ A}{\alpha \sqrt{d}}\right)^{\ell} \ones
\quad =\quad 1+\frac{Z_k}{\alpha} +\frac{R_k}{\alpha^2}
\end{equation}
where 
$$
Z_k=\ee_k^{\tran} \left( \frac{\Delta \circ A}{\sqrt{d}}\right) \ones\qquad \text{and}\qquad R_k=\ee_k^{\tran} \sum_{\ell=2}^\infty \frac 1{\alpha^{\ell-2}}\left( \frac{\Delta\circ A}{\sqrt{d}}\right)^\ell\ones\, .
$$
Notice that the $Z_k$'s are i.i.d. ${\mathcal N}(0,1)$ random variables and denote by $\check{M}=\min_{k\in [n]} Z_k$.

Eq. \eqref{eq:decomp} immediatly yields 
\begin{equation}\label{eq:system}
\left\{ 
\begin{array}{lcl}
\min_{k\in[n]} x_k& \ge & 1+\frac 1{\alpha} \check M +\frac 1{\alpha^2} \min_{k\in[n]} R_k\ ,\\
\\ \min_{k\in[n]} x_k& \le & 1+\frac 1{\alpha} \check M +\frac 1{\alpha^2} \max_{k\in[n]} R_k\, .
\end{array}\right.
\end{equation}
Let $\alpha^*_n= \sqrt{2\log n}$, $\beta^*_n = \alpha^*_n - \frac 1{2\alpha^*_n}  \log ( 4\pi \log n)$ and denote by $G(x)= e^{-e^{-x}}$ the cumulative distribution of a Gumbel distributed random variable. Then it is well-known, see for instance \cite[Theorem 1.5.3]{bib14}, that 
\begin{eqnarray}
\mathbb{P} \left( \alpha^*_n(\check M_n +\beta^*_n) \ge x\right) &\xrightarrow[n\to\infty]{}& G(x)\label{eq:MIN}\, .
\end{eqnarray}
By taking into account this convergence, we can rewrite \eqref{eq:system} as
\begin{align}
1+\frac{\alpha^*_n}{\alpha_n} \left( -1+o_P(1) +\frac {\min_{k\in[n]} R_k}{\alpha^*_n \alpha_n} \right) & \le\min_{k\in[n]} x_k \quad \\
& \le 1+\frac{\alpha^*_n}{\alpha_n} \left( -1+o_P(1) +\frac {\max_{k\in[n]} R_k}{\alpha^*_n \alpha_n} \right)\,  \nonumber.
\end{align}\label{eq:useful}
where we used $(\alpha^*_n)^{-1} (\check M + \beta^*_n)=o_P(1)$. Theorem \ref{th:sparse-matrix} will then follow from the following lemma.

\begin{lemma}\label{lemma:main} Under the assumptions of Theorem \ref{th:sparse-matrix}, the following convergence holds
	$$
	\frac{\max_{k\in [n]} R_k}{\alpha_n \sqrt{2\log n}}\  \xrightarrow[n\to\infty]{\mathcal P} \ 0\qquad \textrm{and}\qquad 
	\frac{\min_{k\in [n]} R_k}{\alpha_n \sqrt{2\log n}} \ \xrightarrow[n\to\infty]{\mathcal P} \ 0 \ .
	$$
\end{lemma}
Proof of Lemma \ref{lemma:main} relies on a careful analysis of the order of magnitude of the extreme values of the remaining term $(R_k)_{k\in[n]}$. 
The sparse structure of matrix $\Delta\circ A$ (either Model (A) or (B)) requires a specific analysis, substantially different from the one in \cite{bib4}.

\section{Proof of Theorem \ref{th:sparse-matrix} for Model (A)} \label{section:proof-main}

We assume that $\Delta_n$ follows Model (A).

In order to prove Lemma \ref{lemma:main}, we first take advantage of the fact that $\| \Delta\circ A/\sqrt{d}\|$ is typically lower than $\kappa$ (see Proposition \ref{prop:spectral-norm}) and replace $R_k$ by a truncated version $\widetilde R_k$ (step 1). We then prove that $A\mapsto \widetilde R_k(A)$ is Lipschitz (step 2). 
The quantity $\widetilde R_k$ being Lipschitz, its centered version is sub-Gaussian if the matrix entries are Gaussian i.i.d. We finally prove that $\widetilde R_k(A)$ is uniformily integrable (step 3). The conclusion easily follows. Although the general strategy is similar to the one developed in \cite{bib4}, the proofs are substantially different. In particular, proofs of step 2 and 3 heavily rely on the block permutation structure of the matrices. 

\subsection{Step 1: Truncation}

Toward proving Lemma \ref{lemma:main}, sub-Gaussiannity is an important property, which follows from Lipschitz properties by standard concentration of measure arguments. Unfortunately $A\mapsto R_k(A)$ fails to be Lipschitz (simply notice that $R_k(A)$ has quadratic and higher order terms). In order to circumvent this issue, we provide a truncated version of $R_k$. 

Let $\kappa>0$ as in Prop. \ref{prop:spectral-norm} (one can take $\kappa=22$), $\eta\in (0,1)$ and $\varphi:\mathbb{R}^+\to [0,1]$ a smooth function:
\begin{equation}\label{eq:def-varphi}
\varphi (x) = \begin{cases}
1 & \text{ if } x \in [0,\kappa+1-\eta] \, ,\\
0 & \text{ if } x \ge \kappa+1 \, 
\end{cases}
\end{equation}
strictly decreasing from $1$ to 0 for $x\in(\kappa+1-\eta,\kappa+1)$. According to Prop. \ref{prop:spectral-norm}, 
$$
\varphi_d(A):= \varphi\left( \left\| \frac {\Delta\circ A}{\sqrt{d}}\right\| \right) 
$$
is equal to one with high probability. We introduce the truncated value:
$$
\tilde{R}_k(A) = \varphi_{d}(A) R_k(A)\ .
$$
We have
\begin{align*}
\mathbb{P} \left( \max_k R_k(A) \neq \max_k \tilde{R}_k(A) \right) & \le
\mathbb{P} \left(\exists k\in [n],\ R_k(A)\neq \tilde{R}_k(A) \right) \\
& \le \mathbb P (\varphi_d(A) < 1) \\
& \le  \mathbb{P}\left( \left\| \frac{\Delta \circ A}{\sqrt{d}}\right\| \ge \kappa\right)\ \xrightarrow[n\to \infty]{}\ 0\, , 
\end{align*}
from which we deduce
\begin{equation} \label{eq:tildeR-property-1}
\frac{\max_{k \in [n]} R_k - \max_{k \in [n]} \tilde{R}_k} {\alpha_n \sqrt{2 \log n}} \xrightarrow[n \to \infty]{\mathcal{P}} 0 \, .
\end{equation}
It is therefore sufficient to prove 
\begin{equation} \label{eq:tildeR-property-2}
\frac{\max_{k \in [n]} \tilde{R}_k} {\alpha_n \sqrt{2 \log n}} \xrightarrow[n \to \infty]{\mathcal{P}} 0 
\end{equation}
to establish the first part of Lemma \ref{lemma:main}. The property of the minimum can be proved similarly.

\subsection{Step 2: Lipschitz property for $\tilde{R}_k(A)$}

For $\ell\ge 2$, we introduce the following summand terms:
\begin{equation}\label{eq:def-summand}
\rho_{k,\ell}(A)  = \ee_k^{\tran} \frac 1{\alpha^{\ell-2}}\left( \frac{\Delta\circ A}{\sqrt{d}}\right)^\ell\ones\qquad \textrm{and}\qquad \tilde{\rho}_{k,\ell}(A)=\varphi_d(A) \rho_{k,\ell}(A)\ ,
\end{equation}
so that $R_k(A) = \sum_{\ell=2}^\infty \rho_{k,\ell}(A)$ and $\tilde R_k(A) = \sum_{\ell=2}^\infty \tilde \rho_{k,\ell}(A)$.

The following lemma is the main result of this section. 

\begin{lemma} \label{lemma:lipsch}
	Let $\kappa>0$ as in Proposition \ref{prop:spectral-norm}, $\delta\in (0,1)$ and $n_0$ such that for all $n\ge n_0$, 
	$$
	\frac{\kappa+1}{\alpha_n}\le 1-\delta\ .
	$$
	For $\ell \ge 2$ and $n\ge n_0$, the function $ \tilde{\rho}_{k,\ell} : \mathcal{M}_n(\mathbb{R}) \rightarrow \mathbb R$ is $K_\ell$-Lipschitz, i.e. 
	\begin{equation}\label{lipschitzrho}
	\left\lvert\tilde{\rho}_{k,\ell} (A) - \tilde{\rho}_{k,\ell} (B)\right\rvert \le K_\ell \left\|A-B\right\|_F\, , \end{equation} 
	where $K_\ell=K_\ell(\kappa,n_0,\delta)>0$ is a constant independent from $k$, $d$ and $n\ge n_0$. Moreover,
	$K:=\sum_{\ell\ge 2} K_{\ell} \ <\ \infty$. 
	In particular, the function $ \tilde{R}_k $ is $K$-Lipschitz : 
	\begin{equation}\label{lipschitz}
	\left\lvert \tilde{R}_k (A) - \tilde{R}_k (B)\right\rvert \le K \left\|A-B\right\|_F\, . \end{equation}
\end{lemma}

Given a $n\times n$ matrix $C$, we define its hermitization matrix $\mathcal{H}(C)$ by: 
$$
\mathcal{H}(C)=\begin{pmatrix}
0 & C \\
C^{\tran} & 0 \\
\end{pmatrix}\, .
$$
A well-known property of ${\mathcal H}(C)$ is its symmetric spectrum and the fact that the singular values of $C$ are the non-negatives eigenvalues of $\mathcal{H}(C)$. In particular, $\left\|C\right\|$ corresponds to the largest eigenvalue of $\mathcal{H}(C)$.

In order to prove Lemma \ref{lemma:lipsch}, we first consider the case where ${\mathcal H}(\Delta\circ A)$ has a simple spectrum, a sufficient condition for the differentiability of $\|\Delta\circ A\|$, we then prove that the Euclidian norm of the gradient of $\tilde{\rho}_{k,\ell}(A)$ is bounded : $ \left\| \nabla \tilde{\rho}_{k,\ell}(A) \right\| \le K_\ell $ and finally proceed by approximation to get the general Lipschitz property.

\begin{proof} We first consider the case where $\mathcal{H}(\Delta\circ A)$ has a simple spectrum. In this case, $\|\Delta\circ A\|$ is equal to the largest eigenvalue of ${\mathcal H}(\Delta\circ A)$ which has multiplicity 1 and is thus differentiable. 
	Denote by $\partial_{ij} = \frac{\partial}{\partial A_{ij}}$. Notice that if $\Delta_{ij}=0$, then for any smooth function $f:\mathbb{R}^{n\times n} \to \mathbb{R}$, $\partial_{ij} f(\Delta\circ A)=0$. If needed, we will take advantage of this property. 
	
	We have :  
	$$
	\left\| \nabla \tilde{\rho}_{k,\ell} (A) \right\| = \sqrt{\sum_{i,j=1}^{n} \left\lvert\partial_{ij}\tilde{\rho}_{k,\ell}(A) \right\rvert^2} \, ,
	$$
	and 
	\begin{eqnarray*}
		\partial_{ij} \tilde{\rho}_{k,\ell} (A) & = & \partial_{ij} \left( \varphi_d(A){\rho}_{k,\ell} (A)\right)\\
		&=&\left(\partial_{ij} \varphi_{d}(A) \right) \rho_{k,\ell}(A) + \varphi_{d} (A) \partial_{ij} \rho_{k,\ell}(A) \quad=: \quad S_{1,ij} + S_{2,ij} \, .
	\end{eqnarray*}
	In particular,
	\begin{equation*}
	\sum_{i,j=1}^{n} \left\lvert \partial_{ij}\tilde{\rho}_{k,\ell}(A) \right\rvert^2 \le 2 \sum_{i,j=1}^{n} \left\lvert S_{1,ij}\right\rvert^2 + 2 \sum_{i,j=1}^{n} \left\lvert S_{2,ij}\right\rvert^2 \, .
	\end{equation*}
	We first evaluate $\sum_{ij} \lvert S_{1,ij}\rvert^2$. 
	
	Recall that $\| \Delta\circ A\|$ being the maximum eigenvalue of ${\mathcal H}(\Delta\circ A)$ which by assumption is simple, it is differentiable by \cite[Theorem 6.3.12]{bib12}. Let $\uu$ and $\vv$ be respectively the left and right normalized singular vectors associated to the largest singular value $\| \Delta\circ A\|$ of $\Delta\circ A$. Then
	$$
	\mathcal{H}(\Delta\circ A)\ww=\left\|\Delta \circ A\right\|\ww \, , \quad  \text{where} \quad  \ww = \begin{pmatrix} 
	\uu \\
	\vv \\
	\end{pmatrix}\, .
	$$
	Notice that $\left\|\ww\right\|^2=2$. We have
	$$
	\partial_{ij} \varphi_d(A)= \frac 1{\sqrt{d}} \varphi'\left( \frac{\|\Delta\circ A\|}{\sqrt{d}}\right)\partial _{ij} \| \Delta\circ A\|
	$$
	and
	\begin{equation}\label{eq:partial-derivative}
	\partial_{ij} \|\Delta\circ A\| = 
	\begin{cases}
	\frac 1{\|w\|}(\uu^{\tran}\ee_i \ee_j^{\tran} \vv + \vv^{\tran}\ee_j \ee_i^{\tran} \uu) = \uu^{\tran}\ee_i \ee_j^{\tran} \vv&\textrm{if}\ \Delta_{ij}\neq 0\ ,\\
	0&\textrm{else.}
	\end{cases}
	\end{equation}
	Let $i\in [n]$. Denote by 
	\begin{equation}\label{def:I}
	{\mathcal I}_i=\{ j\in [n],\ \Delta_{ij}=1\}\, ;
	\end{equation}
	notice that $\textrm{card}({\mathcal I}_i)=d$. We have
	\begin{eqnarray*}
		\sum_{i,j\in [n]} \left\lvert S_{1,ij}\right\rvert^2 & = &\sum_{i\in [n]} \sum_{j\in {\mathcal I}_i} \left\lvert\uu^{\tran}\ee_i \ee^{\tran}_{j}\vv \varphi ' \left( \left\| \frac{\Delta\circ A}{\sqrt d}\right\|\right) \frac{\mathbf{1}}{\sqrt{d}}\rho_{k,\ell}(A) \right\rvert^2  \, ,\\
		& \le& \left\lvert\varphi ' \left( \left\| \frac{\Delta\circ A}{\sqrt d}\right\|\right) \frac{\mathbf{1}}{\sqrt{d}} \rho_{k,\ell}(A) \right\rvert^2 \sum_{i\in [n]} \left\lvert\uu^{\tran}\ee_i\right\rvert^2 \sum_{j\in[n]} \left\lvert\ee^{\tran}_j \vv\right\rvert^2 \, , \\
		& = & \left\lvert\varphi ' \left( \left\| \frac{\Delta\circ A}{\sqrt d}\right\|\right) \frac{\mathbf{1}}{\sqrt{d}} \rho_{k,\ell}(A) \right\rvert^2 \, .
	\end{eqnarray*}
	We now focus on 
	$$
	\left\lvert\frac{\mathbf{1}}{\sqrt{d}} \rho_{k,\ell}(A)\right\rvert^2  = \left\lvert\ee_k^{\tran}  \frac{\mathbf{1}}{\alpha^{\ell-2}}  \left(\frac{\Delta \circ A}{\sqrt{d}}\right)^{\ell}   \frac{\ones}{\sqrt{d}}\right\rvert^2 \, .
	$$
	Notice that $\|\mathbf{1}/\sqrt{d}\|=\sqrt{n/d}$. Since matrix $\Delta\circ A$ follows Model (A), one can notice that $(\Delta\circ A)^{\ell}$ remains a block matrix with only $d$ nonzero terms per row (and per column as well). This property is fundamental for the remaining estimates and fully relies on the Model (A) assumption.
	
	Denote by 
	\begin{equation}\label{def:J}
	{\mathcal J}_{k,\ell}=\left\{ p\in [n],\ \left[ (\Delta\circ A)^{\ell}\right]_{kp} \neq 0\right\}
	\end{equation}
	and by $\mathbf{1}^{{\mathcal J}_{k,\ell}}$ the $n\times 1$ vector with zero coordinates except those belonging to ${\mathcal J}_{k,\ell}$, set to 1. In particular, $\| \mathbf{1}^{{\mathcal J}_{k,\ell}}\| =\sqrt{d}$. Then 
	$$
	\ee_k^{\tran} (\Delta\circ A)^{\ell} \mathbf{1} = \ee_k^{\tran} (\Delta \circ A)^{\ell} \mathbf{1}^{{\mathcal J}_{k,\ell}}\ .
	$$
	We have 
	\begin{eqnarray*}
		\left\lvert\frac{1}{\sqrt{d}} \rho_{k,\ell}(A)\right\rvert^2 & = & \left\lvert\ee_k^{\tran}  \left(\frac{\Delta \circ A}{\alpha \sqrt d}\right)^{\ell-2} \left(\frac{\Delta \circ A}{\sqrt{d}}\right)^2 \frac{\ones}{\sqrt{d}}\right\rvert^2 \, ,\\
		& = & \left\lvert\ee_k^{\tran}  \left(\frac{\Delta \circ A}{\alpha \sqrt d}\right)^{\ell-2} \left(\frac{\Delta \circ A}{\sqrt{d}}\right)^2  \frac{\ones^{{\mathcal J}_{k,\ell}}}{\sqrt{d}}\right\rvert^2 \, ,\\
		& \le &\left\|\ee_k^{\tran}\right\|^2 \left\|\left(\frac{\Delta \circ A}{\alpha \sqrt d}\right)^{\ell-2} \right\|^2 \left\|\frac{\Delta \circ A}{\sqrt{d}}\right\|^4  \left\|\frac{\ones^{{\mathcal J}_{k,\ell}}}{\sqrt{d}}\right\|^2 \, , \\
		& \le & \left\|\frac{\Delta \circ A}{\alpha \sqrt d}\right\|^{2(\ell-2)} \left\|\frac{\Delta \circ A}{\sqrt{d}}\right\|^4\, .
	\end{eqnarray*}

	Using the fact that $\varphi ' \left( \left\| \frac{\Delta\circ A}{ \sqrt d}\right\|\right)=0$ if $\left\| \frac{\Delta \circ A}{\sqrt{d}} \right\| \ge \kappa+1$, we have 
	\begin{align*}
	\left\lvert\varphi ' \left( \left\| \frac{\Delta\circ A}{\sqrt d}\right\|\right) \frac{1}{\sqrt{d}} \rho_{k,\ell} (A) \right\rvert^2 & \le \quad \left\lvert\varphi ' \left( \left\| \frac{\Delta\circ A}{\sqrt d}\right\|\right)\right\rvert^2 \left\|\frac{\Delta \circ A}{\alpha \sqrt d}\right\|^{2(\ell-2)} \left\|\frac{\Delta \circ A}{\sqrt{d}}\right\|^4 \, ,\\
	& \le \quad \left\|\varphi '\right\|_\infty^2 \left(1-\delta\right)^{2(\ell-2)} \, (1+\kappa)^4\ ,
	\end{align*}
	and finally
	\begin{equation}\label{eq:S1}
	\sum_{i,j=1}^{n} \left\lvert S_{1,ij}\right\rvert^2 \le \left\|\varphi '\right\|_\infty^2 \left(1-\delta\right)^{2(\ell-2)} \times (1+\kappa)^4 \,.
	\end{equation} 
	We now evaluate $\sum_{i,j=1}^{n} \left\lvert S_{2,ij}\right\rvert^2 = \sum_{i,j\in [n]} \left\lvert \varphi_d(A) \partial_{ij} \rho_{k,\ell} (A)\right\rvert^2$.
	
	Recall the definitions of ${\mathcal I}_i$ and ${\mathcal J}_{k\ell}$ introduced in \eqref{def:I}, \eqref{def:J}. We have
	$$
	\partial_{ij} \rho_{k,\ell}(A) = 
	\frac{1}{\alpha^{\ell-2}(\sqrt{d})^\ell} \sum_{p=0}^{\ell-1} \ee_k^{\tran} (\Delta \circ A)^p \ee_i \ee_j^{\tran} (\Delta \circ A)^{\ell-1-p} \ones 
	\quad \textrm{if}\quad j \in \mathcal{I}_i
	$$
	and zero else. Then

	\begin{eqnarray}
	\sum_{i\in [n]}  \sum_{j\in \mathcal{I}_i} \left\lvert \partial_{ij} \rho_{k,\ell}(A)\right\rvert^2 & \le & \frac{\ell}{\alpha^{2(\ell-2)}d^\ell}  \left(\sum_{i\in[n]} \sum_{j\in \mathcal{I}_i} \left\lvert  \ee_k^{\tran} (\Delta \circ A)^{\ell-1} \ee_i \ee_j^{\tran} \ones \right\rvert^2\right. \nonumber \\
	& & \quad \left. + \sum_{p=0}^{\ell-2} \sum_{i=1}^{n} \sum_{j\in \mathcal{I}_i} \left\lvert  \ee_k^{\tran} (\Delta \circ A)^p \ee_i \ee_j^{\tran} (\Delta \circ A)^{\ell-1-p} \ones \right\rvert^2  \right)\, , \nonumber \\
	& = & \frac{\ell}{\alpha^{2(\ell-2)}d^\ell}  \left(d \sum_{i\in [n]} \left\lvert  [(\Delta \circ A)^{\ell-1}]_{k,i}  \right\rvert^2 \right. \nonumber\\
	& & \quad \left. + ~ d \sum_{p=0}^{\ell-2} \sum_{i\in [n]} \sum_{j\in \mathcal{I}_i} \left\lvert  [(\Delta \circ A)^p]_{k,i}\, \ee_j^{\tran} (\Delta \circ A)^{\ell-1-p}	\frac{\ones}{\sqrt{d}} 	\right\rvert^2 \right) \, , \nonumber \\
	& \le & \frac{\ell}{\alpha^{2(\ell-2)}d^{\ell-1}} \left(\left[ (\Delta \circ A)^{\ell-1}\left((\Delta \circ A)^{\ell-1}\right)^{\tran}\right]_{k,k} \right. \label{estimate:S2} \\
	& & \quad \left. +  \sum_{p=0}^{\ell-2} \sum_{i\in [n]} \left\lvert  [(\Delta \circ A)^p]_{k,i} \right\rvert^2 \sum_{j\in \mathcal{I}_i} \left\lvert  \ee_j^{\tran} (\Delta \circ A)^{\ell-1-p} \frac{\ones}{\sqrt{d}} \right\rvert^2	\right) \, \nonumber . 
	\end{eqnarray}

	We concentrate on the term $T=\sum_{j\in \mathcal{I}_i} \left\lvert  \ee_j^{\tran} (\Delta \circ A)^{\ell-1-p} 
	\frac{\ones}{\sqrt{d}} 
	\right\rvert^2$ and prove that
	\begin{equation}\label{eq:tricky-estimate}
	T \quad \le\quad  \| \Delta \circ A\|^{2(\ell-1-p)}\ .
	\end{equation}
	Let $I_{{\mathcal I}_i}=\textrm{diag}(\mathbf{1}^{\mathcal{I}_i}(k);\, k\in [n]\})$, where $\mathbf{1}^{\mathcal{I}_i}$ is the $n\times 1$ vector with component 1 if it belongs to $\mathcal{I}_i$ and zero else, then
	$$
	T=\frac{\ones^{\tran}}{\sqrt{d}} \left[(\Delta \circ A)^{\ell-1-p} \right]^{\tran}  I_{\mathcal{I}_i} (\Delta \circ A)^{\ell-1-p} 
	\frac{\ones}{\sqrt{d}} \ .
	$$
	Notice that $(\Delta \circ A)^{\ell-1-p}$ has the form $(P_\tau\otimes \ones_d\ones_d^{\tran})\circ B$ for some $\tau\in {\mathcal S}_m$ and 
	some $n\times n $ matrix $B$. In particular, taking into account the matching between the indices of $I_{\mathcal{I}_i}$ and $(\Delta \circ A)^{\ell-1-p}$'s blocs, there exists a $d\times d$ bloc of matrix $(\Delta\circ A)^{\ell-p-1}$ say $B_i$ such that matrix 
	$$
	\left[(\Delta \circ A)^{\ell-1-p} \right]^{\tran}  I_{\mathcal{I}_i} (\Delta \circ A)^{\ell-1-p} 
	$$
	is zero except a $d\times d$ bloc $B_i^{\tran} B_i$ on the diagonal and
	$$
	T=\frac{\ones_d^{\tran}}{\sqrt{d}} B_i^{\tran} B_i
	\frac{\ones_d}{\sqrt{d}}\  \le\  \| B_i^\tran B_i\|\ \le\ \|B_i\|^2\ \le\ \left\| (\Delta\circ A)^{\ell-p-1}\right\|^2 
	\	\le\ \left\| \Delta\circ A\right\|^{2(\ell -p -1)} \, .
	$$
	Eq.\eqref{eq:tricky-estimate} is established. Notice in particular that the estimate does not depend on the index $i$.
	Plugging this estimate into \eqref{estimate:S2} yields 
	\begin{eqnarray*}
		\lefteqn{\sum_{i\in [n]} \sum_{j\in \mathcal{I}_i} \left\lvert\partial_{ij} \rho_{k,\ell}(A)\right\rvert^2}\\ 
		& \quad  \le& \frac{\ell}{\alpha^{2(\ell-2)}d^{\ell-1}} \left(\left\|(\Delta \circ A)^{\ell-1}\right\|^2 \right. \\
		& & \quad \left. + \sum_{p=0}^{\ell-2} \left[\left((\Delta \circ A)^p\right)^*(\Delta\circ A)^p\right]_{kk} \left\|\Delta \circ A\right\|^{2(\ell-p-1)} \right) \, ,\\
		& \quad \le &\frac{\ell}{\alpha^{2(\ell-2)}d^{\ell-1}} \left(\left\|\Delta \circ A\right\|^{2(\ell -1)} + \sum_{p=0}^{\ell-2} \left\|\Delta \circ A\right\|^{2p} \left\|\Delta \circ A\right\|^{2(\ell-p-1)}  \right) \, , \\
		& \quad = &\frac{\ell^2}{\alpha^{2(\ell-2)}d^{\ell-1}} \, \left\|\Delta \circ A\right\|^{2(\ell-1)} \quad= \quad\ell^2 \left\|\frac{\Delta \circ A}{\alpha \sqrt{d}}\right\|^{2(\ell-2)} \left\|\frac{\Delta \circ A}{\sqrt{d}}\right\|^2\, .\\
	\end{eqnarray*}
	Multiplying by $\lvert\varphi_d(A)\rvert^2$ finally yields the appropriate estimates:
	\begin{eqnarray}\label{eq:S2}
	\sum_{i,j\in [n]} \left\lvert S_{2,ij}\right\rvert^2 & \le & \ell^2 \left\lvert\varphi_{d,\sigma}(A)\right\rvert^2 \left\|\frac{\Delta \circ A}{\alpha \sqrt{d}}\right\|^{2(\ell-2)}\left\|\frac{\Delta \circ A}{\sqrt{d}}\right\|^2 \, ,\nonumber\\
	&\le & \ell^2 (1-\delta)^{2(\ell -2)}(1+\kappa)^2\, .
	\end{eqnarray}
	Combining \eqref{eq:S1} and \eqref{eq:S2}, we obtain :
	\begin{eqnarray} \label{eq:differentiability-estimate}
	\left\| \nabla \tilde{\rho}_{k,\ell} (A) \right\|  & \le &  \sqrt{2 \sum_{i,j=1}^{n} \left\lvert S_{1,ij}\right\rvert^2 + 2 \sum_{i,j=1}^{n} \left\lvert S_{2,ij}\right\rvert^2} \, , \nonumber \\
	& \le &	2(1-\delta)^{\ell-2} (\kappa+1)^2 (\left\|\varphi' \right\|_\infty + \ell)  \ =: K_{\ell}\, .
	\end{eqnarray}
	where $K_\ell$ does not depend upon $k,n,d$ and is summable.
	
	So far, we have established a local estimate over $\left\|\nabla \tilde{\rho}_{k,\ell}(A)\right\|$ for any matrix $A$ such that $\mathcal{H}(\Delta\circ A)$ has a simple spectrum. We first establish the Lipschitz estimate \eqref{lipschitzrho} for two such matrices $A$ and $B$. 
	
	Let $A$, $B$ such that $\mathcal{H}(\Delta\circ A)$ and $\mathcal{H}(\Delta\circ B)$ have simple spectrum and consider the interpolation matrix
	$$A_t = (1-t)A + tB$$ for $t \in [0;1]\ .$
	The continuity of the eigenvalues implies that there exists $\epsilon>0$ sufficiently small such that $\mathcal{H}\left(\Delta\circ {A_t}\right)$ has a simple spectrum for $t \le \epsilon$ and $t\ge 1-\epsilon$. By an argument in \cite[Chapter 2.1]{bib13}, the number of eigenvalues of $\mathcal{H}\left(\Delta\circ {A_t}\right)$ remains constant for $t \in [0,1]$, except maybe for a finite number of points $(t_l; 1\le l\le L)$ : $t_0 = 0 < t_1 < \dots < t_L < t_{L+1} = 1$. Since $\mathcal{H}\left(\Delta\circ {A_t}\right)$ has simple spectrum for $t \in [0;\epsilon) \cup (1-\epsilon;1]$, it has simple spectrum for all $t \notin \{t_l,~ l\in [L]\}$. We can now proceed:
	\begin{eqnarray*}
		\left\lvert\tilde{\rho}_{k,\ell}\left(A_{t_1}\right)-\tilde{\rho}_{k,\ell} \left(A\right)\right\rvert & = &\left\lvert \lim\limits_{\tau \nearrow t_1} \int_{0}^{\tau} \frac{\mathrm{d}}{\mathrm{d} t}\tilde{\rho}_{k,\ell} \left(A_t\right) \mathrm{d}t \right\rvert \\
		& = & \left\lvert \lim\limits_{\tau \nearrow t_1} \int_{0}^{\tau} \nabla \tilde{\rho}_{k,\ell} \left(A_t\right) \circ \frac{\mathrm{d}}{\mathrm{d} t}(A_t) \mathrm{d}t \right\rvert \, ,\\
		& \le &\lim\limits_{\tau \nearrow t_1} \int_{0}^{\tau} \left\|\nabla \tilde{\rho}_{k,\ell} \left(A_t\right) \right\| \times \left\|B-A\right\|_F \mathrm{d}t \le K_\ell \, t_1 \left\|B-A\right\|_F\, .
	\end{eqnarray*}
	By iterating the process over the intervals $(t_{l-1},t_l)$, we get
	\begin{eqnarray*}
		\left\lvert\tilde{\rho}_{k,\ell}\left(B\right)-\tilde{\rho}_{k,\ell} \left(A\right)\right\rvert & \le &\sum_{l=1}^{L+1} \left\lvert\tilde{\rho}_{k,\ell}\left(A_{t_l}\right)-\tilde{\rho}_{k,\ell} \left(A_{t_{l-1}}\right)\right\rvert \, , \\	
		& \le &\sum_{l=1}^{L+1} K_\ell \left(t_l-t_{l-1}\right) \left\|B-A\right\|_F = K_\ell \left\|B-A\right\|_F\,.
	\end{eqnarray*}
	Hence the Lipschitz property along the segment $[A,B]$.
	
	To go beyond, we proceed by density and prove that for a given matrix $\Delta$ as in Model (A), the set of matrices $(\Delta\circ A)$ such that ${\mathcal 
		H}(\Delta\circ A)$ has a simple spectrum is dense in the set of matrices $(\Delta\circ A,\ A\in \mathbb{R}^{n\times n})$.
	
	Let $P_\sigma$ be the permutation matrix used to define $\Delta$ in \eqref{eq:block-permutation} and $I_d$ the identity matrix of size $d$. We define the following $n\times n$ matrices 
	\begin{equation} \label{eq:blocdiag}
	\Pi = P_{\sigma} \otimes I_d \qquad \text{and} \qquad  D_A=(\Delta \circ A) \Pi^{\tran}\, .
	\end{equation}
	Notice that $\Pi$ is a $n\times n$ permutation matrix and that $D_A$ is a block diagonal matrix with $d\times d$ blocks on the diagonal. Since $\Pi\, \Pi^{\tran} =\Pi^{\tran} \Pi=I_n$, we also have 
	$$
	D_A\, \Pi = \Delta \circ A\, .
	$$
	In the framework of Example \ref{example}, matrices $\Pi$ and $D_A$ are given by:
	$$
	\Pi=\begin{pmatrix}
	I_d & 0 & 0 & 0 \\ 0 & 0 & 0 & I_d \\ 0 & I_d & 0 & 0 \\ 0 & 0 & I_d & 0 
	\end{pmatrix}\qquad \text{and}
	\qquad D_A =  \begin{pmatrix} A^{(1)} & 0 & 0 & 0 \\ 0 & A^{(2)} & 0 & 0 \\ 0 & 0 & A^{(3)} & 0 \\ 0 & 0 & 0 & A^{(4)}  \end{pmatrix}\, .
	$$

	An important feature of $D_A$ is that $\Delta\circ A$ and $D_A$ have the same singular values:
	$$
	D_A D_A^{\tran} = (\Delta\circ A) \,\Pi^{\tran}\, \Pi\, (\Delta \circ A)^{\tran} = (\Delta\circ A) (\Delta \circ A)^{\tran} \ ,
	$$
	hence ${\mathcal H}(\Delta\circ A)$ and ${\mathcal H}(D_A)$ have the same eigenvalues and their spectrum, if simple, is simultaneously simple.
	Denote by $(A_{(\mu)})_{\mu\in [m]}$ the $m$ diagonal $d\times d$ blocks of matrix $D_A$ and consider their SVD
	$$
	A_{(\mu)} = U_{(\mu)} \Lambda_{(\mu)} V_{(\mu)}\, .
	$$
	Consider a simultaneous $\varepsilon$-perturbation of the $\Lambda_{(\mu)}$'s into $\Lambda_{(\mu)}^{\varepsilon}$ so that all the $\Lambda_{(\mu)}^{\varepsilon}$'s have distinct diagonal elements, $\varepsilon$-close to the $\Lambda_{(\mu)}$'s. Denote by 
	$$
	A_{(\mu)}^{\varepsilon} = U_{(\mu)} \Lambda^{\varepsilon}_{(\mu)} V_{(\mu)}\, .
	$$
	and let $D^\varepsilon_A$ be the block diagonal matrix with blocks $(A_{(\mu)}^{\varepsilon})_{\mu\in [m]}$. Then ${\mathcal H}(D^\varepsilon_A)$ is arbitrarily close to ${\mathcal H}(D_A)$ and has a simple spectrum. Note that $D^\varepsilon_A\Pi$ 
	is $\varepsilon$-close to $\Delta\circ A$, is such that ${\mathcal H} ( D^\varepsilon_A\Pi )$ has a simple spectrum and 
	has the same pattern as $\Delta\circ A$ in the sense that:
	$$
	\Delta_{ij}=0\quad \Rightarrow \quad \left( D^\varepsilon_A\Pi \right)_{ij} =0\, .
	$$
	To emphasize this property, we introduce the $n\times n$ matrix $A^\varepsilon $ defined as 
	$$
	[A^\varepsilon]_{ij}= \begin{cases}
	[ D^\varepsilon_A\Pi]_{ij} & \text{if}\ \Delta_{ij}=1\, ,\\
	A_{ij}&\text{else}
	\end{cases}
	$$
	so that 
	$$\| \Delta\circ A^\varepsilon - \Delta \circ A\|_F = \|A^\varepsilon- A\|_F \xrightarrow[\varepsilon\to 0]{} 0\, .$$ 
	
	We can now conclude. Let $\Delta\circ A$, $\Delta\circ B$ be given and $D^\varepsilon_A \Pi=\Delta \circ A^\varepsilon$ and $ D^\varepsilon_B\Pi=\Delta\circ B^\varepsilon$ constructed as previously;
	notice that $C\mapsto \tilde \rho_{k,\ell}(C)$ is continuous. Then 
	\begin{eqnarray*}
		\left\lvert\tilde{\rho}_{k,\ell}(B)-\tilde{\rho}_{k,\ell}(A)\right\rvert & \le & \left\lvert\tilde{\rho}_{k,\ell} \left(B^\varepsilon\right)-\tilde{\rho}_{k,\ell}(B)\right\rvert + K_\ell \left\|B_\epsilon - A_\epsilon\right\|_F + \left\lvert\tilde{\rho}_{k,\ell} \left(A_\epsilon\right)-\tilde{\rho}_{k,\ell}(A)\right\rvert  \, ,\\
		& \xrightarrow[\epsilon \to 0]{} &K_\ell \left\|B-A\right\|_F\,.
	\end{eqnarray*}
	This concludes the proof of the Lipschitz property.
\end{proof}	

\subsection{Step 3: uniform estimate for $\mathbb{E} \tilde{R}_k(A)$}
As a consequence of the Lipschitz property of $\widetilde R_k$, $\widetilde R_k(A)$ if centered is sub-Gaussian if $A$ is a $n\times n$ matrix with i.i.d. ${\mathcal N}(0,1)$ entries. The following estimate easily follows using Tsirelson-Ibragimov-Sudakov inequality (\cite[Theorem 5.5]{bib5}). 
\begin{proposition}\label{prop:subG}
	Under the assumptions of Lemma \ref{lemma:lipsch}, the following estimate holds true:
	$$
	\mathbb{E}\max_{k\in [n]} \left(\tilde{R}_k - \mathbb{E}\tilde{R}_k\right) \le K \sqrt{2\log n}\, .
	$$
\end{proposition}
For the proof, see \cite[Proposition 2.3]{bib4}.

The rest of the section is devoted to the control of $\mathbb{E}\widetilde R_k(A)$.

\begin{proposition} \label{prop:control_exp}
	Under the assumptions of Theorem \ref{th:sparse-matrix}, there exists $n_1\in \mathbb{N}$ and a constant $C>0$ such that for all $n\ge n_1$,
	$$
	\sup_{k\in [n]}\left\lvert \mathbb{E} \tilde{R}_k(A) \right\rvert \ \le\ C\, .
	$$
\end{proposition}

\begin{proof} 
	Recall that $n=d\times m$ and that $\Delta\circ A$ is a block permutation matrix with $m$ blocks $(A^{(\mu)})_{\mu\in [m]}$ of size $d\times d$. We choose a given block $A^{(\mu)}$ and denote by $\mu_1,\cdots, \mu_d$ the $d$ indices corresponding to the rows of block $A^{(\mu)}$ in $\Delta\circ A$. By exchangeability, we have 
	$$
	\forall k \in [d]\,,\quad  \mathbb{E} \tilde{R}_{\mu_k}(A) = \mathbb{E} \tilde{R}_{\mu_1}(A)\, .
	$$
	Denote by $\ones^{(\mu)}$ the $n\times 1$ vector with ones for the indices $(\mu_i)_{i\in [d]}$ and zeros elsewhere. We have
	\begin{eqnarray}
	\left\lvert\mathbb{E} \widetilde{R}_{\mu_k}(A) \right\rvert & = &\left\lvert\frac{1}{d} \sum_{i=1}^{d} \mathbb{E} \widetilde{R}_{\mu_i}(A)\right\rvert \nonumber \\
	& = &\left\lvert\frac{1}{d} \sum_{i=1}^{d} \mathbb{E} \left(\varphi_{d} (A) \ee_{\mu_i}^\tran \left(\frac{\Delta\circ A}{\sqrt{d}}\right)^2 \left(I - \frac{\Delta \circ A}{\alpha \sqrt{d}}\right)^{-1} \ones \right)\right\rvert\, ,\nonumber \\
	& = &\left\lvert\frac{1}{d} \mathbb{E} \left(\varphi_{d} (A)  \ones^{(\mu)\tran} \left(\frac{\Delta \circ A}{\sqrt{d}}\right)^2 \left(I - \frac{\Delta \circ A}{\alpha \sqrt{d}}\right)^{-1} \ones \right)\right\rvert \, ,\nonumber\\
	& \le &\mathbb{E} \left\lvert\varphi_{d} (A)  \frac{\ones^{(\mu)\tran}}{\sqrt{d}} \left(\frac{\Delta \circ A}{\sqrt{d}}\right)^2 \left(I - \frac{\Delta \circ A}{\alpha \sqrt{d}}\right)^{-1} \frac{\ones}{\sqrt{d}}\right\rvert\, .\label{eq:mean-estimate}
	\end{eqnarray}
	We start by expanding $\left(I - \frac{\Delta \circ A}{\alpha \sqrt{d}}\right)^{-1}$ :
	$$
	\left\lvert\varphi_{d} (A)  \frac{\ones^{(\mu)\tran}}{\sqrt{d}} \left(\frac{\Delta \circ A}{\sqrt{d}}\right)^2 \left(I - \frac{\Delta \circ A}{\alpha \sqrt{d}}\right)^{-1} \frac{\ones}{\sqrt{d}}\right\rvert = \left\lvert\varphi_{d} (A) \sum_{\ell=2}^{\infty} \frac{\ones^{(\mu)\tran}}{\alpha^{\ell-2}\sqrt{d}} \left(\frac{\Delta \circ A}{\sqrt{d}}\right)^\ell \frac{\ones}{\sqrt{d}}\right\rvert\, .
	$$
	Notice that $(\Delta \circ A)^\ell$ is a block matrix constituted of $m$ blocks of size $d\times d$. In particular, among the $d$ row 
	$$
	\left( \left[ (\Delta\circ A)^\ell\right]_{ij} \right)_{i\in \{\mu_1,\cdots, \mu_d\},\, j\in [n]} \ ,
	$$
	there exist $\nu_1,\cdots, \nu_d$ (consecutive) indices such that the only non-null entries are
	$$
	\left( \left[ (\Delta\circ A)^\ell\right]_{ij} \right)_{i\in \{\mu_1,\cdots, \mu_d\},\, j\in \{\nu_1,\cdots, \nu_d\}}\, . 
	$$
	Denote by $\mathbf{1}^{(\nu)}$ the $n\times 1$ vector of ones for the indices $(\nu_i)_{i\in [d]}$ and zeroes elsewhere. As a consequence of the previous remark, 
	$$
	\mathbf{1}^{(\mu)\tran} (\Delta\circ A)^{\ell} \mathbf{1} = \mathbf{1}^{(\mu)\tran} (\Delta\circ A)^{\ell} \mathbf{1}^{(\nu)}
	$$
	and
	\begin{eqnarray*}
		\frac{1}{\alpha^{\ell-2}}\left\lvert\frac{\ones^{(\mu)\tran}}{\sqrt{d}} \left(\frac{\Delta \circ A}{\sqrt{d}}\right)^\ell \frac{\ones}{\sqrt{d}}\right\rvert & =& 
		\frac{1}{\alpha^{\ell-2}}\left\lvert\frac{\ones^{(\mu)\tran}}{\sqrt{d}} \left(\frac{\Delta \circ A}{\sqrt{d}}\right)^\ell \frac{\ones^{(\nu)}}{\sqrt{d}}\right\rvert \\
		& \le & \frac{1}{\alpha^{\ell-2}} \left\|\frac{\ones^{(\mu)}}{\sqrt{d}}\right\|  \left\|\frac{\Delta \circ A}{\sqrt{d}}\right\|^{\ell} \left\|\frac{\ones^{(\nu)}}{\sqrt{d}}\right\| \, ,\\
		& \le &\left\|\frac{\Delta \circ A}{\alpha\sqrt{d}}\right\|^{\ell-2} \left\|\frac{\Delta \circ A}{\sqrt{d}}\right\|^2\, .
	\end{eqnarray*}
	Let $\kappa>0$ as in Proposition \ref{prop:spectral-norm}, $\delta\in(0,1), n_0\in \mathbb{N}$ as in Lemma \ref{lemma:lipsch}, then 
	\begin{eqnarray*}
		\sum_{\ell=2}^{\infty} \frac{1}{\alpha^{\ell-2}} \left\lvert\varphi_{d} (A) \frac{\ones^{(\mu)*}}{\sqrt{d}} \left(\frac{\Delta \circ A}{\sqrt{d}}\right)^\ell \frac{\ones}{\sqrt{d}}\right\rvert 
		& \le &\varphi_{d} (A) \sum_{\ell=2}^{\infty} \left\|\frac{\Delta \circ A}{\alpha\sqrt{d}}\right\|^{\ell-2} \left\|\frac{\Delta \circ A}{\sqrt{d}}\right\|^2 \, ,\\
		& = &\varphi_{d} (A) \left\|\frac{\Delta \circ A}{\sqrt{d}}\right\|^2 \sum_{\ell=0}^{\infty} \left\|\frac{\Delta \circ A}{\alpha\sqrt{d}}\right\|^{\ell} \\
		& \le & (1+\kappa)^2 \sum_{\ell=2}^\infty (1-\delta)^\ell\, ,\\
		&\le & \frac{(1+\kappa)^2}{\delta}  \ .
	\end{eqnarray*}
	Plugging this estimate into \eqref{eq:mean-estimate} concludes the proof of the estimation of $\lvert\mathbb{E}\widetilde R_{\mu_k}(A)\rvert$. This estimate being uniform over $\mu_1,\cdots,\mu_d$ and over all the blocks $(A^{(\mu)})$, the proposition is proved.

\end{proof}

\subsection{Proof of lemma \ref{lemma:main}} \label{sec:proof_lemma_main}
Combining Lemma \ref{lemma:lipsch}, Propositions \ref{prop:subG} and \ref{prop:control_exp} one can prove Lemma \ref{lemma:main} as in \cite[Section 2.3]{bib4} with minor adaptations.

\section{Proof of Theorem \ref{th:sparse-matrix} for Model (B)} \label{section:proof-main-B}

We assume that $\Delta_n$ follows Model (B). 

The strategy of proof closely follows the one in \cite{bib4}, with one specific issue to handle: the uniform bound on $\mathbb{E} \widetilde R_k$.
An important property exploited in \cite{bib4} to establish a uniform bound over $\mathbb{E} \widetilde R_k$ was the exchangeability of the $\widetilde R_k$'s (or block exchangeability in the case of Model (A)). There is not enough structure in Model (B) to guarantee this exchangeability (which might not hold). 

We carefully address this issue hereafter.

\subsection{A uniform bound over $\mathbb{E} \widetilde R_k$ for Model (B)}\label{sec:issue2}

\begin{proposition} \label{prop:control_exp_modelB}
	Under the assumptions of Theorem \ref{th:sparse-matrix}, uniformly in $k \in [n]$,
	$$
	\mathbb{E} \tilde{R}_k = \mathcal{O} \left( \frac{\alpha}{\sqrt{d}} \right) \, .
	$$
\end{proposition}
Proof of Proposition \ref{prop:control_exp_modelB} relies on two important facts.
\begin{itemize}
	\item The fact that almost surely ${\mathcal H}(\Delta\circ A)$ has a simple spectrum, hence the Lipschitz function $\| \Delta\circ A\|$ is almost surely differentiable with an explicit formula for the partial derivatives, see \eqref{eq:partial-derivative}. Details are provided in Appendix \ref{app:simplicity}. 
	\item The Gaussian integration by parts (i.b.p.) formula: If $Z\sim {\mathcal N}(0,1)$ then $\mathbb{E}\, Zf(Z) =\mathbb{E} f'(Z)$. Interestingly, this formula holds for $f$ Lipschitz. In this case, $f$ is absolutely continuous hence almost surely differentiable (see for instance \cite[Chap. 7, Thm. 4]{bib24}) with linear growth at infinity. 
\end{itemize}
Recall that $\varphi_d(A)=\varphi\left( \frac{\|\Delta\circ A\|}{\sqrt{d}}\right)$, where $\varphi$ is defined in \eqref{eq:def-varphi}.
\begin{proof}
	In order to get an asymptotic bound over $\mathbb{E} \widetilde R_k(A)$, we expand its expression:
	\begin{eqnarray*}
		\mathbb{E}\widetilde R_k(A) & = & \mathbb{E} \left[\varphi_d(A) \ee_k^\tran \left( \frac{\Delta \circ A}{\sqrt{d}}\right)^2 Q \ones\right] \, , \\
		& = & \frac{1}{d} \sum_{i \in \mathcal{I}_k} \sum_{j \in [n]} \mathbb{E} \left[\varphi_d(A) (\Delta \circ A)_{ki} \left( (\Delta \circ A) Q\right)_{ij}\right] \, , \\
		& = & \frac{\alpha}{\sqrt{d}} \sum_{i \in \mathcal{I}_k} \sum_{j \in [n]} \mathbb{E} \left[\varphi_d(A) (\Delta \circ A)_{ki} \left( - \delta_{ij} +Q_{ij}\right) \right] \, , \\
		& = & -\frac{\alpha}{\sqrt{d}} \sum_{i \in \mathcal{I}_k} \mathbb{E} \left[\varphi_d(A) (\Delta \circ A)_{ki} \right] + \frac{\alpha}{\sqrt{d}} \sum_{i \in \mathcal{I}_k} \sum_{j \in [n]} \mathbb{E} \left[\varphi_d(A) (\Delta \circ A)_{ki} Q_{ij} \right] \, .
	\end{eqnarray*}
	At this point, we use the Gaussian i.b.p. formula applied to $A\mapsto \varphi_d(A)$ which is Lipschitz and a.s. differentiable with explicit derivative (see \eqref{eq:partial-derivative}).
	\begin{eqnarray*}
		\mathbb{E}\widetilde R_k(A) & = & -\frac{\alpha}{\sqrt{d}} \sum_{i \in \mathcal{I}_k} \mathbb{E} \left[\partial_{ki} \varphi_d(A) \right] + \frac{\alpha}{\sqrt{d}} \sum_{i \in \mathcal{I}_k} \sum_{j \in [n]} \mathbb{E} \left[\partial_{ki} \left(\varphi_d(A) Q_{ij} \right)\right] \, , \\
		& = & -\frac{\alpha}{d} \sum_{i \in \mathcal{I}_k} \mathbb{E} \left[ u_kv_i\varphi' \left( \frac{\|\Delta\circ A\|}{\sqrt{d}} \right) \right] + \frac{\alpha}{d} \sum_{i \in \mathcal{I}_k} \sum_{j \in [n]} \mathbb{E} \left[u_kv_i\varphi' \left( \frac{\|\Delta\circ A\|}{\sqrt{d}}\right)Q_{ij} \right] \\
		&&\qquad  + \frac{\alpha}{\sqrt{d}} \sum_{i \in \mathcal{I}_k} \sum_{j \in [n]} \mathbb{E} \left[\varphi_d(A) \left( \partial_{ki}Q_{ij} \right)\right]\, , \\
		& =: & T_1 + T_2 + T_3\, .
	\end{eqnarray*}
	We first handle the term $T_1$ by Cauchy-Schwarz inequality:
	\begin{eqnarray*}
		\left\lvert T_1 \right\rvert & \le & \frac{\alpha}{d} \mathbb{E} \left\lvert u_k \sum_{i \in \mathcal{I}_k} v_i \varphi'\left(\frac{\|\Delta\circ A\|}{\sqrt{d}} \right) \right\rvert \, , \\
		& \le &  \frac{\alpha}{d} \mathbb{E} \left[ \sqrt{d} \left\|\boldsymbol{v}\right\| \left\lvert\varphi'\left(\frac{\|\Delta\circ A\|}{\sqrt{d}}\right)\right\rvert \right] \quad =\quad \mathcal{O} \left( \frac{\alpha}{\sqrt{d}}\right).
	\end{eqnarray*}
	We now handle the term $T_2$ :
	\begin{eqnarray*}
		\left\lvert T_2 \right\rvert & \le & \frac{\alpha}{d} \mathbb{E} \left[ \left\lvert \varphi' \right\rvert \left\lvert \sum_i \sum_j v_iQ_{ij} \right\rvert\right] \, , \\
		& \le & \frac{\alpha}{d} \mathbb{E} \left[ \left\lvert \varphi'\right\rvert \cdot \left\lvert \boldsymbol{v}^* Q  \ones \right\rvert \right] 
		\ \le \ \frac{\alpha}{\sqrt{d}} \sqrt{\frac{n}{d}} \mathbb{E} \left[ \left\lvert\varphi'\right\rvert \left\| Q \right\| \right] \quad =\quad \mathcal{O} \left( \frac{\alpha}{\sqrt{d}}\right).
	\end{eqnarray*}
	We finally handle the term $T_3$. Notice that $\partial_{ki} Q_{ij} = \frac 1{\alpha\sqrt{d}} Q_{ik} Q_{ij}$ and denote by $\boldsymbol{\omega}:=(Q_{ik}1_{{\mathcal I}_k})_{i \in [n]}$. Notice that $\|\boldsymbol{\omega}\|^2 \le \ee^*_k Q^* Q\ee_k$
	hence $\| \boldsymbol{\omega}\| \le \| Q\|$ and 
	\begin{eqnarray*}
		\left\lvert T_3 \right\rvert & = & \frac{1}{d} \left\lvert{\mathbb{E}\left[ \varphi_d(A) \boldsymbol{\omega}^* Q \ones \right] }\right\rvert\quad \le \quad \frac{1}{d} \mathbb{E}\left[ \varphi_d(A) \|\boldsymbol{\omega}\|\, \|Q\|\, \|\ones\| \right] \, ,  \\
		& \le &  \frac{\sqrt{n}}{d} \mathbb{E} \left[ \varphi_d(A) \left\|Q\right\|^2 \right]\ =\ {\mathcal O}\left(\frac 1{\sqrt{d}}\right)\, .
	\end{eqnarray*}
	Combining these asymptotic notations finally yields :
	\begin{equation*}
	\mathbb{E} \widetilde R_k(A) = \mathcal{O} \left( \frac{\alpha}{\sqrt{d}} \right)\,.
	\end{equation*}
\end{proof}
Notice that even if the bound obtained in Proposition \ref{prop:control_exp_modelB} is weaker than the one obtained in Proposition \ref{prop:control_exp} or in \cite[Prop. 2.4]{bib4}, it is still sufficient to establish the feasibility under Model (B).

\section{Proofs of Theorem \ref{th:global stability} and Proposition \ref{prop:stability}}\label{section:proof-stability}

\subsection{Proof of Theorem \ref{th:global stability}}

The proof is a combination of Takeuchi and Adachi's theorem \cite[Theorem 3.2.1]{bib19} and Proposition \ref{prop:spectral-norm}. We first recall the definition of Volterra-Liapunov stability, see for instance \cite[Section 3.2]{bib19}: Let $B$ be a $n\times n$ real matrix.
$B$ is \textit{\textbf{Volterra-Liapunov stable}} if there exists a $n\times n$ positive definite diagonal matrix $D$ such that $DB+B^{\tran}D$ is negative definite.

Going back to Eq. \eqref{eq:LV}, according to Takeuchi and Adachi's theorem \cite[Th. 3.2.1]{bib19}, this LV system has a unique nonnegative and globally stable equilibrium if $M_n-I_n$ is Volterra-Liapunov stable. 

We now rely on the asymptotic spectral properties of $M_n$ to study the Volterra-Liapunov stability of $M_n-I_n$. We drop the subscript $n$ in the sequel. Take $D=I$ then $$
D(M-I)+(M-I)^{\tran}D=M+M^{\tran}-2I$$ is an hermitian matrix. This matrix is negative definite if all its eigenvalues are negative. Given that $M+M^{\tran}$ is also hermitian, we just have to check that the spectral radius $\rho\left(M+M^{\tran}\right)<2$. According to Proposition \ref{prop:spectral-norm}:
$$
\mathbb{P} \left(\rho \left(M+M^{\tran}\right) < 2 \right)\quad  \ge\quad  \mathbb{P} \left(\left\|M \right\|<1 \right) \quad \xrightarrow[n \to +\infty]{}\quad  1.
$$
Thus, the probability that $M-I$ is Volterra-Liapunov stable converges to 1 as $n\to\infty$. By \cite[Th. 3.2.1]{bib19}, this implies that the probability that the LV system \eqref{eq:LV} has a unique nonnegative and globally stable equilibrium converges to 1 as $n\to\infty$.

\subsection{Proof of Proposition \ref{prop:stability}}

We first prove the first part of the proposition. By Theorem \ref{th:global stability}, there exists a unique nonnegative globally stable equilibrium to \eqref{eq:LV}. If there exists $\epsilon >0$ such that eventually $\alpha_n \ge (1+\epsilon) \alpha_n^*$ where $\alpha_n^*=\sqrt{2\log n}$, then this equilibrium $\bs{x}_n$ is positive by Theorem $\ref{th:sparse-matrix}$ with overwhelming probability as $n\to \infty$.

The rest of the proof closely follows the proof of \cite[Corollary 1.4]{bib4} and is omitted.

\section{Conclusion}\label{sec:conclusion}

In this article we study the feasibility and stability of sparse large ecosystems modelled by a large Lotka-Volterra system of coupled differential equations:
$$
\frac{d\,\boldsymbol{x}_n}{dt}  = \boldsymbol{x}_n(\boldsymbol{1}_n - \boldsymbol{x}_n + M_n \boldsymbol{x}_n)\, .
$$

Our work is motivated by recent research \cite{bib6} which suggests that in the light of many ecological and biological datasets living networks are often sparse. It also illustrates the interest to study feasibility in relation with the normalization of the interaction matrix's entries beyond the non-sparse full i.i.d. models, and opens perspectives to study models with more structure such as elliptic interactions or patch models.

In the model under investigation, the interaction matrix $M_n$ is a sparse random matrix, where the sparsity is encoded by a patterned matrix $\Delta_n$ based on an underlying $d_n$-regular graph, and the randomness by i.i.d. random variables (matrix $A_n$) for non-null entries. The single parameter $d_n$ of the regular graph provides an easy one-dimensional parametrization of the connectance of the foodweb. 

Our main conclusion is that beyond the standard normalization $1/\sqrt{d_n}$ of the interaction matrix $\Delta\circ A$, which guarantees a bounded norm 
$$
\left\| \frac{\Delta \circ A}{\sqrt{d}}\right\|={\mathcal O}_P(1)\ ,
$$
an extra factor $1/\alpha_n$ with $\alpha_n\to \infty$ is needed to reach feasibility. The interaction matrix finally writes
$$
M_n=\frac{\Delta_n\circ A_n}{\alpha_n\sqrt{d_n}} 
$$
and a sharp phase transition occurs at $\alpha_n^*=\sqrt{2\log(n)}$. Interestingly, the same phase transition as in the non-sparse case occurs. 

\begin{figure}[h]
	\centering
	\includegraphics[scale=0.7]{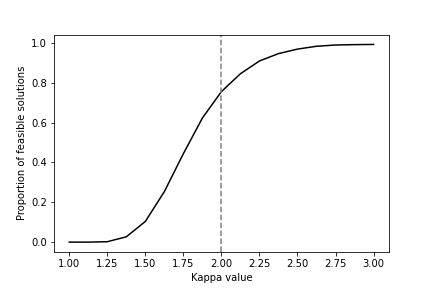}
	\caption{Let $n=15000$, $d=10$ (notice that $d\ge \log(n)\simeq 9.61$). Matrix $\Delta_n$ is drawn at random once for all among the adjacency matrices of $d$-regular graphs (and a priori does not follow Model (A)). Each point of the curve represents the proportion of feasible solutions $\bs{x}_n$ of Eq. \eqref{eq:equilibrium-A} over $1500$ realizations of random matrices $A_n$ for different values of $\kappa$, with $\alpha_n=\sqrt{\kappa \log(n)}$. The phase transition resemble those of Figure \ref{fig:simu-model-A}.} \label{fig:non-struct-simu}
\end{figure}

In the sparse setting $\log(n)\le d_n \ll n$, we rely on an extra block-structure assumption over matrix $\Delta_n$, namely Model (A), to establish the feasibility and the phase transition. Our method of proof crucially relies on this technical assumption which somehow concentrates the non-null entries of the sparse interaction matrix (and its powers) into localized blocks. 

However simulations (cf. Fig \ref{fig:non-struct-simu}) suggest that this block structure assumption is not necessary and could be relaxed. Hence the following:
\begin{open}
	Let $\Delta_n$ the adjacency matrix of a deterministic $d_n$regular graph, with $d_n\ge \log(n)$, and $A_n$ a random matrix with i.i.d. ${\mathcal N}(0,1)$ entries. Consider the equation 
	$$
	\boldsymbol{x}_n =\boldsymbol{1}_n +\frac{\Delta_n\circ A_n}{\alpha_n \sqrt{d_n}} \boldsymbol{x}_n\, ,\quad \alpha_n \to \infty\, .
	$$
	Is it true that the same phase transition as in Theorem \ref{th:sparse-matrix} holds?
\end{open}

\begin{appendices}

	\section{With probability one, the singular values of a sparse random matrix are distinct}\label{app:simplicity}
	We establish hereafter that with probability one the singular values of matrix $\Delta\circ A$ are distinct, a key argument in the proof of Proposition \ref{prop:control_exp_modelB} to compute the partial derivatives of $A\mapsto \|\Delta \circ A\|$.  
	
	The lemma below and its proof are inspired by Nick Cook \cite{bib28}, whom we thank for his help.
	
	\begin{lemma}[Cook \cite{bib28}]\label{lemma:cook}
		Let $n\ge 1$, $A_n$ a $n\times n$ matrix with i.i.d. ${\mathcal N}(0,1)$ entries and $\Delta_n$ the adjacency matrix of a $d$-regular graph. Then with probability one, all the singular values of $\Delta_n\circ A_n$ are distinct.
	\end{lemma}
	\subsubsection*{Remark}
	The original statement of Cook is slightly more general: matrix $A_n$ entries only need a distribution with positive density, and the deterministic matrix $\Delta_n$ only needs a generalized diagonal, i.e. $(\Delta_{i\sigma(i)};\, i\in [n])$ for some $\sigma\in {\mathcal S}_n$, with $n-1$ non null entries. 
	\begin{proof}
		Let $\mathcal{E}_\Delta$ be the set of matrices with entries supported on the nonzero entries of $\Delta$,
		\begin{equation*}
		\mathcal{E}_\Delta = \left\{ \Delta \circ X\, ; \  X=(X_{ij}) \in \mathbb{R}^{n\times n} \right\}.
		\end{equation*}
		Thus, $\mathcal{E}_\Delta$ is the support of the law of $\Delta\circ A$. Besides, $\mathcal{E}_\Delta$ is a variety as a subspace of $\mathbb{R}^{n \times n}$. 
		
		Let $\mathcal{R}$ denote the set of matrices with a repeated singular value. It is the set of $n\times n$ matrices $X$ for which the characteristic polynomial $p$ of $X^{\tran}X$ has zero discriminant ($\rho$), see for instance \cite[Section 3.3.2]{bib29}.
		$$
		\mathcal{R} \ = \ \left\{ X \in \mathbb{R}^{n\times n}\,;\  \rho \left( p \left(X^{\tran}X \right) \right) = 0\right\}\ 
		= \ \left\{ X \in \mathbb{R}^{n\times n}\,;\  P \left( X \right) = 0\right\} \, ,
		$$
		where $P:\mathbb{R}^{n \times n} \rightarrow \mathbb{R}$ defined by $P(X)=\rho(p(X^{\tran}X))$ is a polynomial in the entries of $X$.
		It follows that $\mathcal{R}$ is an algebraic variety in $\mathbb{R}^{n \times n}$.
		
		Hence, $\mathcal{E}_\Delta \cap \mathcal{R}$ is either equal to $\mathcal{E}_\Delta$, or a subvariety of $\mathcal{E}_\Delta$ of zero Lebesgue measure (under the product measure on $\mathcal{E}_{\Delta}$). 
		
		For the claim, it suffices to show that $\mathcal{E}_\Delta \not\subset \mathcal{R}$ hence to exhibit $Y\in \mathcal{E}_\Delta$ with distinct singular values. By Birkhoff's theorem \cite[Theorem 8.7.2]{bib12}, the doubly stochastic matrix $\Delta/d$ writes
		$$
		\frac{\Delta}{d} = \sum_{\sigma\in {\mathcal S}_n} a_\sigma P_{\sigma}\ , \quad a_{\sigma}\ge 0 \quad \textrm{and}\quad \sum_{\sigma\in {\mathcal S}_n} a_\sigma=1\, ,
		$$
		where $P_\sigma$ is the permutation matrix associated to $\sigma\in {\mathcal S}_n$. There exists in particular $\sigma^*$ with $a_{\sigma^*}>0$ and $P_{\sigma^*}\in {\mathcal E}_{\Delta}$. Let $P_{\sigma^*}=(P_{ij})_{i,j\in [n]}$ then matrix $Y=(iP_{ij})_{i,j\in [n]}$ has distinct singular values $(1,\cdots, n)$. This completes the proof.

	\end{proof}
	
\end{appendices}

\bibliography{sn-biblio.bib}

\end{document}